\documentclass[letterpaper, 10 pt, conference]{ieeeconf}  

\IEEEoverridecommandlockouts                              
\input{cdc_input}
\usepackage{mathtools}
\usepackage{amsmath}
\usepackage{cuted}

\DeclareMathOperator{\Id}{I}
\DeclareMathOperator*{\argmax}{arg\,max}



\usepackage{hyperref}

\title{\LARGE \bf
Inferring System and Optimal Control Parameters of \\ Closed-Loop Systems from Partial Observations
}

\author{Victor Geadah\textsuperscript{1,$\dagger$}, Juncal Arbelaiz\textsuperscript{2}, Harrison Ritz\textsuperscript{3}, Nathaniel D. Daw\textsuperscript{3}, \\ Jonathan D. Cohen\textsuperscript{3}, Jonathan W. Pillow\textsuperscript{1,3}
\thanks{\textsuperscript{1}Program in Applied and Computational Mathematics, Princeton University, Princeton NJ, UA}%
\thanks{\textsuperscript{2}Department of Mechanical and Aerospace Engineering, Princeton University, Princeton NJ, US}
\thanks{\textsuperscript{3}Princeton Neuroscience Institute, Princeton University, Princeton NJ, USA}
\thanks{\textsuperscript{$\dagger$}Correspondence: {\tt victor.geadah@princeton.edu}}
}

\begin{document}

\maketitle
\thispagestyle{empty}
\pagestyle{empty}

\begin{abstract}
We consider the joint problem of system identification and inverse optimal control for discrete-time stochastic Linear Quadratic Regulators. 
We analyze finite and infinite time horizons in a partially observed setting, where the state is observed noisily.  To recover closed-loop system parameters, we develop inference methods based on probabilistic state-space model (SSM) techniques.
First, we show that the system parameters exhibit non-identifiability in the infinite-horizon from closed-loop measurements, and we provide exact and numerical methods to disentangle the parameters. Second, to improve parameter identifiability, we show that we can further enhance recovery by either (1) incorporating additional partial measurements of the control signals or (2) moving to the finite-horizon setting. 
We further illustrate the performance of our methodology through numerical examples. 

\end{abstract}

\section{Introduction}

An important problem in the analysis of scientific data is accurately inferring the properties of dynamical systems based on partial observations. 
This is particularly important in neuroscience, where the neural systems that support the highly complex computations of interest are both difficult to observe, which is done for instance through neuro-imaging \cite{knosche2022eeg} or electrophysiology, and are often unknown a-priori. 
Probabilistic state-space models provide a powerful language to model neuroscientific data \cite{Wu2006, Yu2006, schimel2022ilqrvae, geadah2024parsing}, however, the fact that this systems are naturally operating in closed-loop configurations is often disregarded in the inference.

A growing body of work uses control-theoretic tools to analyze the dynamics of neural systems (see e.g., \cite{ RevModPhys.90.031003, Kao2019-mn, ATHALYE2020145}). 
For example, recent work used a standard formulation in optimal control, the \textit{Linear Quadratic Regulator} (LQR), to model neural dynamics in  motor cortex during movement preparation \cite{kao2021optimal}. While
posing  neuroscientific questions and analyzing related data from a control-theoretic lens is promising, 
inferring the properties of controllers from closed-loop neural recordings 
poses particular statistical challenges \cite{Tsiamis2023}. 

Work in inverse optimal control (IOC)\textemdash the problem of inferring the cost function of an optimal controller from recordings with known system dynamics\textemdash has historically focused on fully observed systems \cite{Kalman1964, Jameson1973, Zhang2019b}.  
Recent work has explored control over partially observed systems \cite{Zhang2019, Zhang2021}, but assumes the system parameters are known. 
Similarly, 
inverse optimal control has been proven useful in understanding
human sensorimotor control, but existing approaches have depended on full observations and/or a known plant \cite{Priess2015,Schultheis2021, Straub2022,  Karg2022}. 
Refs.~\cite{pmlr-v28-golub13, 10.7554/eLife.10015}
 explored system parameter inference from a partially-observed neural system, but focused on system identification rather than inferring the parameters of the underlying control process.


Unlike engineered or known sensorimotor systems, for which 
accurate models of the dynamics might be available, the dynamics of neural systems are often unknown. Furthermore, they entail closed-loop processes. 
Leveraging control-theoretic perspectives on such un-observed processes recently gained traction in providing strong expressivity in state-space models \cite{schimel2022ilqrvae}. Similarly, \textit{model-free} \cite{Clarke2022}, or \textit{direct}, approaches to IOC involve no knowledge of the system and bypass it altogether, and offer discernible strengths for approximate control \cite{pmlr-v87-clavera18a, bertsekas2011dynamic}. In both cases, the study of the underlying joint parameters and identifiability is relatively unexplored. 

Motivated by this problem, we focus on identifying both system parameters (e.g., system dynamics) and optimal control parameters (e.g., cost functions) from partial observations, assuming the system is operating in closed-loop with an optimal controller corresponding to the stochastic LQR.
Our approach consists of two steps. First, we leverage the analytic and tractable expectation-maximization (EM) algorithm for system identification of linear-Gaussian state space models  \cite{Ghahramani1996} to estimate a system's closed-loop dynamics. Second, we develop methods for inferring LQR state and action cost functions from these estimates of the closed-loop dynamics. 

Our contributions are three-fold. First in Section~\S\ref{s:inf}, we provide a characterization of the non-identifiability encountered in the infinite horizon setting, as a necessary and sufficient approach to identifiability. Next, we provide two different alternatives to improve parameter identifiability, namely i) by considering the finite-horizon setting in \S\ref{s:finite}; and ii) by using partial measurements of the control inputs in \S\ref{s:inf+control}. 
In all sections, we provide both theoretical treatments and either exact or numerical methods to disentangle the parameters.

\subsection{Notation}
For a sequence of vectors $\{x_t\}_{t=1}^T$ where $x_t \in \R^n$, we abbreviate the notation as $x_{1:T}$. 
We assume that all random variables are defined on some probability space $(\Omega, \mathcal{B}, \mathbb{P})$, and for a general random variable $y:\Omega \to \R^n$ we abuse notation by denoting its instance $y \in \R^n$ and write its probability density function (p.d.f) as $p(y)$.
%
%
We denote by $\E_{p(x)}\left[f(x)\right] = \int f(x) p(x) dx$ the expected value of $f(x)$ for the random variable $x$.  
We denote by $\Id_n$ the $n \times n$ identity matrix. 
Finally, \textit{identifiability} refers to the following:


\begin{definition}[Identifiability, {\cite[Def.~5.2]{LehmCase98}}]
    If $X$ is distributed according to $p_\theta$, then $\theta$ is said to be \textit{unidentifiable on the basis of X} if there exist $\theta_1 \neq \theta_2$ for which $p_{\theta_1} = p_{\theta_2}$.
\end{definition}

\section{Setting and Problem Formulation}

\subsection{Control Problem and Assumptions}

We consider the stochastic linear quadratic regulator (LQR) optimal control problem, which with a finite and known time-horizon $T > 0$ reads 
\begin{subequations}\label{eqs:LQR_problem}
\begin{multline}
    \min_{x_{1:T}, u_{0:T-1}} J_T(x_{1:T}, u_{0:T-1}) =  \\
        \frac{1}{T}\E\left[x_T^\top Q_T x_T + \sum_{t=0}^{T-1} \left(x_t^\top Q x_t + u_t^\top R u_t\right)\right] \label{eq:LQR_cost}
\end{multline}
subject to 
\begin{align}
    x_{t+1} &= A x_t + B u_t + \omega_t, \quad \omega_t \sim  \mathcal{N}(0,W) \label{eq:system:xt}
\end{align}
\end{subequations}
for $t \in \{0, \dots, T-1\}$, $x_0 \sim \mathcal{N}(0, W_0)$, for the state $x_t \in \R^n$ and control inputs $u_t \in \R^m$, where $A \in \R^{n \times n}$ and $B \in \R^{n \times m}$. The system is assumed to be driven by zero-mean Gaussian noise with covariance $W$. 
The LQR problem aims to minimize the cost function $J_T$, with parameters $Q, Q_T \in \R^{n \times n}$, positive semi-definite matrices, and $R \in \R^{m \times m}$, positive definite. The expectation is taken w.r.t. $\mathbb{P}\left(x_0, \omega_{0:T-1}\right)$, over the initial condition $x_{0}$ and the dynamics noise sequence $\omega_{0:T-1}$.

For simplicity we assume there is not a cross-term between control input and state in the cost function. The system is assumed to have access to \textit{full} and \textit{noiseless} measurements of its own state for feedback control. (That is, there is no need to implement a Kalman filter for control). Furthermore, for well-posedeness and uniqueness of control gain solutions below, we assume that $(A,B)$ is controllable, $(A,C)$ is observable \cite{Kalman1960}. 

It is well-known \cite{anderson2007optimal} that the optimal control inputs $u^*$ for the problem \eqref{eqs:LQR_problem} are in \textit{proportional feedback} form
\begin{align}
    u^*_t &= -K_t x_t \nonumber \\
    &\coloneqq -(B^{\top} P_{t+1} B + R)^{-1}(B^{\top} P_{t+1} A) x_t \label{eq:K_t}
\end{align}
for the control gain $K_t \in \R^{m \times n}$, determined from the Riccati recursion term 
\begin{multline}
     P_t =  Q + A^{\top} P_{t+1} A \\ - A^{\top} P_{t+1} B (B^{\top} P_{t+1} B + R)^{-1} B^{\top} P_{t+1} A\label{eq:riccati}
\end{multline}
and $P_T = Q_T$. This yields the closed-loop system
\begin{align}\label{eq:closed_loop_finite}
    x_{t+1} = \underbrace{(A - B K_t)}_{F_t} x_t + \omega_t
\end{align}
with time varying closed-loop linear dynamics $F_t \in \R^{n \times n}$. We depict the block diagram for the closed-loop plant in Fig.~\ref{fig:system_schematics}. Furthermore, we will also be considering the infinite horizon setting, where we take $T \to \infty$ in \eqref{eq:LQR_cost}, removing $Q_T$. In this case, the optimal control signal $u^*_t$ is a linear feedback of the state as well, 
 but with a \textit{static} gain 
\begin{align}\label{eq:K_ss}
    u^*_t &= -K_{ss} x_t = -(B^{\top} P_{ss} B + R)^{-1}(B^{\top} P_{ss} A) x_t
\end{align}
for the control gain $K_{ss}$, where $P_{ss}$ satisfies the discrete algebraic Riccati equation (DARE)
\begin{multline}\label{eq:DARE}
     P_{ss} =  Q + A^{\top} P_{ss} A \\ - A^{\top} P_{ss} B (B^{\top} P_{ss} B + R)^{-1} B^{\top} P_{ss} A.
\end{multline}
This now yields the closed-loop system
\begin{align}\label{eq:closed_loop_infinite}
    x_{t+1} = \underbrace{(A - B K_{ss})}_{F_{ss}} x_t + \omega_t
\end{align}
with time-invariant linear dynamics $F_{ss} \in \R^{n \times n}$.

\begin{figure}[t]
    \centering
    \includegraphics[width=\columnwidth]{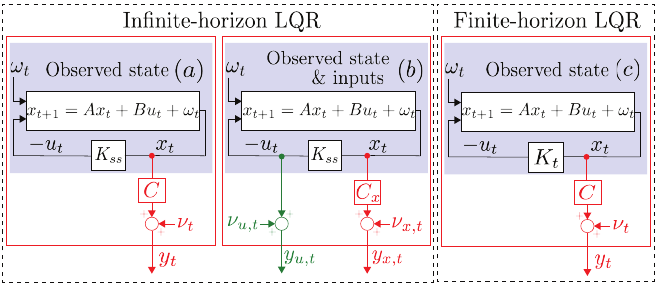}
    \caption{Models of the optimal closed-loop dynamics and observations in the different scenarios analyzed in this work. The infinite-horizon LQR is displayed in (a)-(b). These panels differ in the observations available for inference to an external observer: in (a) the observer only has access to partial and noisy measurements of the state, while in (b) noisy measurements of the control input are also accessible. (c) Finite-horizon LQR and observation model.}
    \label{fig:system_schematics}
\end{figure}

Importantly, we consider the setting in which an external observer has access only to partial and noisy measurements $y_t \in \R^{d}$ of the state $x_t$ for inference. In particular, we consider linear Gaussian observations
\begin{align}
     y_t &= C x_t + \nu_t, \quad \nu_t \sim \mathcal{N}(0,V) \label{eq:system:yt}
\end{align}
for $t \in \{0, \dots, T\}$, where $C \in \R^{d \times n}$ and for the covariance matrix $V$. 
The system \eqref{eq:system:xt}-\eqref{eq:system:yt} is sometimes referred to as an \emph{input-driven linear Gaussian state-space model} (LG-SSM \cite{murphy2012machine} (see non-input driven definition in Appendix~\ref{app:model_equivalence}), also referred to as \emph{Linear Dynamical System} (LDS) \cite{roweis1999unifying}, where specifically the control input sequence $u_{0:T-1}$ and corresponding $x_{1:T}$ trajectory are optimal under the stochastic LQR framework.

Having introduced the problem setting, we end with some notation. Let our set of parameters $\Theta$ be decomposed into $\Theta_s$ the set of system parameters $\theta_s = \{A,B,Q,R\} \in \Theta_s$, $\Theta_o$ the set of output parameters $\theta_o = \{C\} \in \Theta_o$, and $\Theta_n$ the set of noise parameters $\theta_n = \{W_0,W,V\} \in \Theta_n$. We have $\Theta = \Theta_s + \Theta_o + \Theta_n$. Furthermore, we will denote by $\mathcal{F}:\Theta_s \to \R^{n \times n}$ (resp. $\mathcal{F}_t$) the function that takes our system parameters $\theta_s$ and returns the closed-loop dynamics $\mathcal{F}(\theta_s) = F_{ss}$ in \eqref{eq:closed_loop_infinite} (resp. $\mathcal{F}_t(\theta_s) = F_t$ in \eqref{eq:closed_loop_finite}). 

\subsection{Problem Formulation}

An external observer has access only to partial observations $y_{0:T}$ of a system under control operating in closed-loop. At a high level, the question we seek to answer in this work is: strictly from observations $y$, what can be inferred about the system parameters $\{A,B,C,W,W_0,V\}$ and the control parameters $\{Q, R\}$? We will articulate most of our results on the specific interaction between $\{A, B, Q, R\}$. Non-identifiability is expected, thus we further ask: what are additional measurements or assumptions that we can use to reach identifiability? 

The basis for identifiability and parameter inference is the \textit{marginal likelihood}, 
\begin{align}\label{eq:marginal_lik}
    p_\theta(\mathbf{y}) =  \prod_{i=1}^N p_\theta\left(y^{(i)}_{0:T}\right),
\end{align}
for a dataset $\mathbf{y} = \lbrace y^{(i)}_{0:T} \rbrace_{i = 1}^N$. 
We perform parameter inference under the closed-loop statistical model $p_\theta$ presented in \eqref{eq:closed_loop_infinite}-\eqref{eq:system:yt} by seeking $\hat\theta \in \Theta$ maximizing the marginal likelihood.

To shed light into the non-identifiability, we consider the first case of the infinite horizon setting. We plot in Fig.~\ref{fig:iden_infinite_motivation}\textbf{A} what such parameter inference approach can yield from observations in this case. We can infer the closed-loop parameters $F_{ss}$, but the parameters $\{A, B, Q, R\}$ are not identifiable. 
This can be explained from the latent dynamics probabilities, which satisfy
\begin{align}
    p_\theta(x_{t+1} \mid x_t, u^*_t) &= p(x_{t+1} \mid x_t, \theta_s, \theta_n) \nonumber \\
    &= p\left(x_{t+1} \mid x_t, \mathcal{F}(\theta_s), \theta_n\right) \label{eq:dynamics_prob:lds_equiv}
\end{align}
and are \textit{equivalent} to the LG-SSM latent dynamics. In 
App.~\hyperref[app:model_equivalence]{A}
(Prop.~\ref{prop:appendix:eq_LDS_LQR}), we formalize that under the assumption that the true model is indeed $p_\theta$, then parameter inference in the general statistical model defined through \eqref{eq:closed_loop_infinite}-\eqref{eq:system:yt} is {equivalent} to parameter inference in an LG-SSM.  Hence performing parameter inference can be reduced to, and is in fact equivalent to, the two following steps:
\begin{enumerate}
    \item Infer the closed-loop parameters $F_{ss}$ from partial observations $\mathbf{y}$; 
    \item Investigate the interaction of the system parameters $\theta_s$ in the closed-loop dynamics $F_{ss}$.
\end{enumerate}
 
In the next section \S\ref{ss:setting_EM} we address the first step and discuss how in practice we maximize the marginal likelihood in \eqref{eq:marginal_lik} and infer the closed-loop system parameters. The second step will be the topic of Sections \S\ref{s:inf} and \S\ref{s:finite} for respectively the infinite- and finite-horizon settings. Finally, in Section~\S\ref{s:inf+control} we explore how the introduction of additional control measurements can help relax this identifiability constraint. See Fig.~\ref{fig:iden_infinite_motivation}\textbf{B} for a schematic of our approach.

\begin{figure}
    \centering
    \includegraphics[width=\columnwidth]{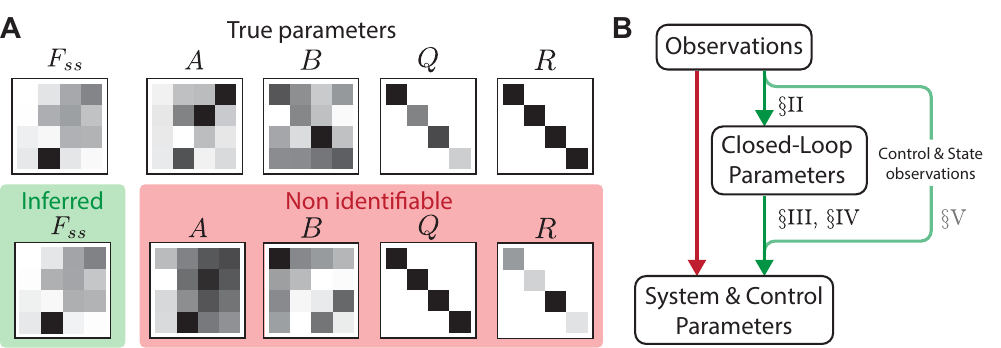}
    \caption{(\textbf{A}) Recovery from observations of the system $\{A,B\}$ and cost function $\{Q,R\}$ parameters in the infinite horizon setting is limited to identification of the closed-loop dynamics $F_{ss}$. We plot the inferred $F_{ss}$ using EM, as well as a combination $\{A,B,Q,R\}$ that yields this same $F_{ss}$.
    (\textbf{B}) Parameter inference can be reduced to closed-loop identification with probabilistic methods, followed by an investigation of the interaction between the different parameters in setting the closed-loop dynamics.}
    \label{fig:iden_infinite_motivation}
\end{figure}

\subsection{Closed-Loop System Recovery}\label{ss:setting_EM}



A general difficulty in fitting parameters to data in state space models is that we must marginalize over latent sequences $x_{0:T}$
\begin{align*}
    p_\theta(y_{0:T}) &= \int p_\theta(y_{0:T} \mid x_{0:T})p_\theta(x_{0:T}) d x_{0:T} \\
    &= \E_{p_\theta(x_{0:T})}\left[p_\theta(y_{0:T} \mid x_{0:T})\right]
\end{align*}
to evaluate the marginal likelihood. For general state space models, this integration cannot be done in closed-form. Furthermore, optimizing this expression requires optimizing parameters of the support distribution of this integral. A popular approach that circumvents the latter consists of using the \emph{expectation-maximization} (EM) algorithm \cite{Shumway1982ANAT, Dempster1977, Ghahramani1996} to maximize \eqref{eq:marginal_lik}. The EM algorithm is centered around the following lower bound
\begin{align}
    \ln p_\theta(\mathbf{y}) \geq \mathcal{Q}\big(\theta \mid \theta'\big) \coloneqq \E_{p(x_{0:T} \mid \mathbf{y}, \theta')} \left[\ln p(x_{0:T}, \mathbf{y} \mid \theta) \right]
    \label{eq:Q:def}
\end{align}
to the marginal \textit{log}-likelihood, for any $\theta, \theta' \in \Theta$. The algorithm can then be summarized succinctly as recursively maximizing 
\begin{align}
    \theta^{(j+1)} \gets \argmax_{\theta} \mathcal{Q}\big(\theta \mid \theta^{(j)}\big)
\end{align}
over iterates $j = 1, 2, \dots$ until convergence. This algorithm has the convenient property that it guarantees monotonic increase in the lower bound $\mathcal{Q}$, but is generally highly dependent on initialization \cite{holmes2013derivation}. We use the EM-algorithm to maximize the likelihood in \eqref{eq:marginal_lik}. Further details for our problem are provided in Appendix \hyperref[app:EM_details]{B}. 
In Fig.~\ref{fig:recovery_comparison}\textbf{B}, we provide estimation errors for the recovery procedures as a function of the number of EM iterations, attesting to its convergence.

In linear Gaussian state-space models, 
the marginal likelihood in \eqref{eq:marginal_lik} can in fact be evaluated in closed-form using Kalman filtering \cite[See \S18.3.1.3]{murphy2012machine}. Nevertheless, the EM-algorithm for the LG-SSM has closed-form updates for the parameters $\theta$ and remains a standard approach for its efficiency and flexibility \cite{holmes2013derivation}. Importantly for us, we will leverage these closed-form updates in \S\ref{s:inf+control} to showcase how the inclusion of control measurements can help us further in the identification of the parameters $\theta$. 

Regardless of the parameter inference methods, since the latent variables $x_t$ are not observed our model is only identifiable up to any invertible transformation of the latent space \cite{Tsiamis2023, Kalman1964}. 
Indeed, if we define a new latent state as $\hat x_t = H x_t$ for $H \in \R^{n \times n}$ invertible, then the closed-loop dynamics can be redefined as $\hat F_{ss} = H F_{ss} H^{-1}$, with dynamics covariance $H W H^\top$.
%
The topic of invariant metrics under non-identifiability of system and control parameters dates back to original work by Kalman \cite{Kalman1964, Kalman1963} and is an active area of research.
Here, we study the relationship of the parameters $\theta_s = \{A,B,Q,R\}$ in setting $F_{ss}$. This helps in practice to provide estimates for $\hat\theta_s = \{\hat A, \hat B, \hat Q, \hat R\}$ from an inferred $\hat F_{ss}$. How identifiability considerations impact the estimates $\hat \theta$ is out of the scope of this paper. Thus, one must exercise caution when comparing parameters \textit{estimated} (e.g. $\hat A$) from this procedure with \textit{true} (e.g. $A$), generative, parameters. 

\section{Identifiability in Infinite Horizon}\label{s:inf}


 The preceding section highlighted how our higher level goal amounts to disentangling the parameters $\{A, B, Q, R\}$ from measurements of the closed-loop system. The key idea is that by having direct measurements of the closed-loop, we will be able to disentangle our control gain equation \eqref{eq:K_t} (resp. \eqref{eq:K_ss}) from our Ricatti term \eqref{eq:riccati} (resp. \eqref{eq:DARE}), yielding a system of two equations. 

Let us first consider the infinite time-horizon setting, where we now presume to have access to a measurement of $F_{ss}$. Strictly from our control gain definition, we have
\begin{align*}
    F_{ss} &= A - B K_{ss} \\
    &= \left(\Id - B(R + B^{\top} P_{ss} B)^{-1} B^{\top} P_{ss} \right) A
\end{align*}
and in particular 
\begin{align}
    A &=  \left(\Id_{n} - B(R + B^{\top} P_{ss} B)^{-1} B^{\top} P_{ss} \right)^{-1} F_{ss} \nonumber \\
    &= \left(\left(\Id_{n} + B R^{-1} B^{\top}  P_{ss}\right)^{-1} \right)^{-1} F_{ss} \nonumber \\
    &= \left(\Id_n + B R^{-1} B^{\top} P_{ss} \right) F_{ss}\label{eq:lin_P:derivation}
\end{align}
where we used the Woodbury matrix identity to perform the inverse. Next, shifting to the DARE in \eqref{eq:DARE}, we note 
\begin{align*}
     P_{ss} &= Q  + A^{\top} P_{ss} A -A^{\top} P_{ss} B \left(R + B^{\top} P_{ss} B\right)^{-1} B^{\top} P_{ss} A \\
    &= Q  + A^{\top} P_{ss} A -A^{\top} P_{ss} B K_{ss} \\
    &= Q + A^\top P_{ss} (A -  B K_{ss})
\end{align*}
such that
\begin{align}
    P_{ss} &= Q + A^\top P_{ss} F_{ss} \label{eq:sylvester_P:derivation}
\end{align}

These two simple derivations together show that the interaction between the different parameters of interest can be reduced to considering a new system of equations, which we make explicit in the following proposition: 

\begin{proposition}\label{prop:param_recovery_sylvester_P}
    Consider the infinite horizon (stochastic) LQR problem with system parameters $\{A, B\}$ and control parameters $\{Q, R\}$. Let $(A,B)$ be stabilizable, and let $F_{ss}$ be the closed-loop dynamics matrix.  
    Then, the unique solution $P$ to the Sylvester equation
    \begin{subequations}\label{eqs:prop:param_recovery_sylvester_P}
    \begin{align}\label{eq:prop_recov:sylvester_P}
        P - A^\top P F_{ss} &= Q
    \end{align}
    satisfies the matrix equation
    \begin{align}\label{eq:prop_recov:lin_P}
         B R^{-1} B^\top  P F_{ss} = A - F_{ss}
    \end{align}
    \end{subequations}
\end{proposition}
\begin{proof} $(A,B)$ stabilizable ensures that $P=P_{ss}$ in \eqref{eq:prop_recov:sylvester_P} is unique, as derived from the DARE. The proof then simply follows from our derivations in \eqref{eq:lin_P:derivation}-\eqref{eq:sylvester_P:derivation}. 
\end{proof}

An interesting implication of the stabilizability of $(A,B)$, is that since $P$ in \eqref{eq:prop_recov:sylvester_P} is unique, this in turn implies that $F^{-1}$ and $-A^\top$ (equivalently $-A$) do not share an eigenvalue. 

In Prop.~\ref{prop:param_recovery_sylvester_P}, $F_{ss}$ is, \textit{a-priori}, the only known parameter. 
As such, to make any conclusions about one or a combination of parameters without neither resorting to $P_{ss}$ nor constraints on our parameters, we need to assume the others to be known.

\paragraph{Solving for $A$} First and foremost, we may wish to solve for $A$ in \eqref{eqs:prop:param_recovery_sylvester_P}. This is slightly more challenging, as the system exhibits second-order tendencies in $A$.  To see this, consider the simplification $B = R = \Id$. In that case, we can easily solve for $P$ in \eqref{eq:prop_recov:lin_P} and substitute it in \eqref{eq:prop_recov:sylvester_P}, and obtain that $A$ satisfies the system \eqref{eqs:prop:param_recovery_sylvester_P} if and only if 
\begin{align}\label{eq:AQF_infinite}
    Q =  \left(A - F_{ss}\right) F_{ss}^{-1} - A^\top \left(A - F_{ss}\right) 
\end{align}

This is a second-order matrix equation in $A$, combining of a second-order self-adjoint term with a first-order term of Sylvester-transpose form. This combination is unfortunately not (readily) amenable to common second-order matrix equation techniques \cite{Crone1981}. Instead, we elect to use an iterative approach that draws from iterative solutions to Ricatti-type second-order equations \cite{Kleinman:1968:iterative_riccati}. 

\begin{proposition}\label{prop:iterative}
    Let $\{B, R, Q\}$ be given, as well as $F_{ss}$. Let $P_j$, $j = 1, 2, \dots$, be the (unique) positive definite solution to the algebraic Sylvester equation 
    \begin{subequations}
    \begin{align}
        P_j - A_j^\top P_j F_{ss} &= Q \label{eq:prop:A_recovery:sylvester}
    \end{align}
    where, recursively,
    \begin{align}\label{eq:prop:A_recovery:lin}
        A_j &= \left(\Id + BR^{-1} B^\top P_{j-1}\right) F_{ss}
    \end{align}
    where $A_0$ is chosen such that it does not share any eigenvalue with $-F_{ss}^{-1}$. Then 
    \begin{enumerate}
        \item $P_{ss}  \leqq P_{j+1} \leqq P_j \leqq \dots$, $j=1,2,\dots$
        \item $\lim_{j \to \infty} A_j = A$
    \end{enumerate}
    \end{subequations}
\end{proposition}
\begin{proof}
This result follows from a similar iterative procedure in \cite{Kleinman:1968:iterative_riccati}. The key idea is that instead of modifying the $F_{ss}$ at every step holding $A$ constant, we let $A$ be adjusted so that $F_{ss}$ stays constant. 
\begin{enumerate}
    \item
    By our assumption on $A_0$, $P_0$ defined as a solution to the Sylvester equation \eqref{eq:prop:A_recovery:sylvester} is unique and well-defined.
    Let $P_j$ be well-defined. Then as $A_{j+1}$ is adjusted by the second step \eqref{eq:prop:A_recovery:lin}, itself derived from the control gain equation, it ensures that there is a control gain $K_{j+1}$ such that $A_{j+1} - B K_{j+1} = F_{ss}$.  This in turn ensures that $(A_{j+1}, B)$ is stabilizable, and that the solution $P_{j+1} \succ 0$ to the DARE \eqref{eq:DARE} of parameters $\{A_{j+1}, B, Q, R\}$ is unique and well defined. Then by induction, the sequence $\{P_j, K_{j+1}\}$ for $j=1,2,\dots$ is well-defined. Furthermore, it satisfies the same recursive form as in \cite{Kleinman:1968:iterative_riccati}, and it follows from  \cite{Kleinman:1968:iterative_riccati} that $P_{j} - P_{j+1} \succeq 0$, $P_j - P_{ss} \succeq 0$. 
    \item Monotonic convergence of positive operators ensures that the limit for $P_j$ exists and the lower bound and uniqueness of $P_{ss}$ yields $\lim_{j\to \infty}P_j = P_{ss}$. Then by Prop.~\ref{prop:param_recovery_sylvester_P}, it holds that $A_j \to A$.
\end{enumerate}
\end{proof}

We require $A_0$ to not share any eigenvalue with $-F_{ss}^{-1}$. Since $(A,B)$ is stabilizable and $F_{ss}$ is invertible, then $F_{ss}$ has eigenvalues of norm less than 1 and greater than 0. Possible initalizations for $A_0$ thus include $A_0 \in \{F_{ss}, \Id, 0\}$.  

The above result provides monotonic and quadratic 
convergence (see Fig.~\ref{fig:recovery_comparison}\textbf{A}) to the true $A$ when we have access to $F_{ss}$, and know the other parameters $\{B,R,Q\}$. 
In practice, one may have only a-priori estimates or imperfect measurements of any of these parameters.  The proof above only holds for the true parameters, thus convergence in this setting is not guaranteed. 
Nonetheless, we have found this iterative method to fare nearly just as well as a numerical root-finding algorithm based on Newton's method, and the proof above can give insights into the well-posedness of the problem and convergence guarantees. 
To assess this sensitivity, we investigated the impact of adding a scalar perturbation $\epsilon \Id$, $\epsilon > 0$, to either $F_{ss}$ or $Q$ on our estimate for $A$. We show in Fig.~\ref{fig:recovery_comparison}\textbf{C}-\textbf{D} that both iterative and root finding methods handle this perturbation similarly. 

\paragraph{Solving for $B$ and  $R$} Next, we notice that $\{B, R\}$ are in direct interaction through $B R^{-1} B^{\top}$. One can solve for $B R^{-1} B^{\top}$ given $\{A, Q, F_{ss}\}$ by solving for the unique $P$ in \eqref{eq:prop_recov:sylvester_P} and substituting it into the second equation \eqref{eq:prop_recov:lin_P}. As for the specific values of $B$ and $R$, they are not strictly identifiable. For instance, for any given $B$ and $R \succ 0$, let $\hat B = B R^{-\frac{1}{2}} U R^{\frac{1}{2}}$ for $U \in \R^{m \times m}$ unitary, arbitrary. Then 
\begin{align*}
    \hat B R^{-1} {\hat B}^\top &= \left(B R^{-\frac{1}{2}} U R^{\frac{1}{2}}\right) R^{-1} \left(B R^{-\frac{1}{2}} U R^{\frac{1}{2}}\right)^\top \\
    &= B R^{-\frac{1}{2}} U U^\top R^{-\frac{1}{2}} B^\top \\
    &= B R^{-1} B^\top,
\end{align*}
making $\hat B$ and $B$ non-identifiable. 

\paragraph{Solving for $Q$} Finally, solving for our last parameter of interest can be done quite easily if $BR^{-1} B^\top$ is invertible, as it allows to solve for $P_{ss}$ in \eqref{eq:prop_recov:lin_P} and subsitute it in \eqref{eq:prop_recov:sylvester_P} to readily read of $Q$. For instance, if $B=R=\Id$, then $Q$ is immediately available from \eqref{eq:AQF_infinite}. If $BR^{-1} B^\top$ is not invertible, for instance, if the control inputs $u_t$ are lower dimensional than the state $x_t$, then care must be taken by considering the pseudo-inverse and identify its null-space.

\begin{figure}
    \centering
    \includegraphics[width=0.9\columnwidth]{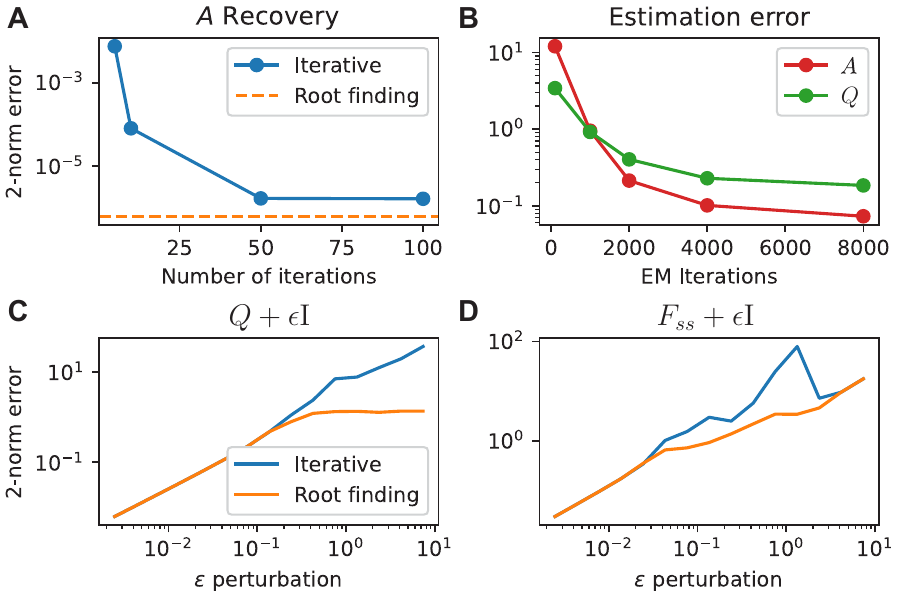}
    \caption{Parameter estimation convergence and robustness properties. (\textbf{A}) Matrix 2-norm error between the true parameter $A$ used for simulation ($4\times 4$ unitary rotation) and the estimated $A$ from the iterative procedure in Prop.~\ref{prop:iterative} as a function of the number of iterations. (\textbf{B}) Similar estimation error to \textbf{A} for $A$ given $Q$ as well as for $Q$ given $A$, this time using $\hat F_{ss}$ estimated from the EM-algorithm, as a function of the number of iterations. We use $C=\Id$ for identifiability. (\textbf{C}-\textbf{D}) Recovery error for $A$ as a function of $Q$ (resp. $F_{ss}$) with perturbed entries. See text for further details.}
    \label{fig:recovery_comparison}
\end{figure}

\section{Partially Observed Finite Horizon}\label{s:finite}

In the \textit{finite} horizon stochastic LQR setting described in \eqref{eqs:LQR_problem}, the closed-loop dynamics satisfy \eqref{eq:closed_loop_finite}. We make the simplifying assumption of considering $Q_T=Q$. Now, just as in the infinite-horizon setting, our system transition probabilities satisfy
\begin{align}
    p_\theta(x_{t-1} \mid x_t, u^*_t) &= p(x_{t-1} \mid x_t, A,B,Q,R, \theta_n) \nonumber\\
    &= p(x_{t-1} \mid x_t, \mathcal{F}_t(A,B,Q,R) = F_t, \theta_n)
\end{align}
for $t \in \{1, \dots, T-1\}$. Importantly, again, the system parameters $\{A,B\}$ and performance metric parameters $\{Q,R\}$ only affect our marginal likelihood through the closed-loop dynamics $F_t$. In turn, this yields a similar system of equations to Prop.~\ref{prop:param_recovery_sylvester_P} for the finite-horizon setting.
\begin{proposition}\label{prop:param_recovery_sylvester_Pt}
    Consider the \textit{finite} horizon (stochastic) LQR problem with system parameters $\{A, B\}$ and control parameters $\{Q, R\}$. Let $(A,B)$ be stabilizable, and let $F_t$ be the closed-loop dynamics for $t \in \{1,\dots,T-1\}$. Then,
    \begin{subequations}\label{eqs:prop:param_recovery_sylvester_finite}
    \begin{align}\label{eq:prop_recov:sylvester_P_finite}
        P_t - A^\top P_{t+1} F_{t} &= Q
    \end{align}
    satisfies the matrix equation
    \begin{align}\label{eq:prop_recov:lin_P_finite}
         B R^{-1} B^\top  P_{t+1} F_t = A - F_t
    \end{align}
    and $P_T = Q$.
    \end{subequations}
\end{proposition}

The main difference with the infinite horizon setting in terms of recovery considerations comes down to the terminal condition $P_T = Q$. This means that we can restrict the solution space over $\Theta_s$ strictly from our last step. Indeed, for $t=T-1$, from \eqref{eq:prop_recov:lin_P_finite} we obtain
\begin{align}
    B R^{-1} B^\top  P_{T} F_{T-1} &= A - F_{T-1} \nonumber\\
    \Longleftrightarrow\quad A &= \left(B R^{-1} B^\top Q + \Id\right)F_{T-1}\label{eq:terminal_FT-1_AQ}
\end{align}
This provides an essential, supplemental, link between the parameters $\{A,B,Q,R\}$. 

To leverage this, consider again $B=R=\Id$. Starting from \eqref{eq:prop_recov:sylvester_P_finite} and substituting in \eqref{eq:prop_recov:lin_P_finite},
\begin{align*}
    P_{T-1} &= Q + A^\top P_T F_{T-1} \\
    \Longleftrightarrow (A - F_{T-2}) F_{T-2}^{-1} &= Q + A^\top (A - F_{T-1})
\end{align*}
and now \eqref{eq:terminal_FT-1_AQ} reads as $A = (Q + \Id)F_{T-1}$, such that
\begin{multline}
     \left(A - F_{T-1}\right) F_{T-1}^{-1} = \\ (A - F_{T-2}) F_{T-2}^{-1} - A^\top \left(A - F_{T-1}\right).
\end{multline}

This expression is similar to \eqref{eq:AQF_infinite} in the infinite horizon setting, but now using our terminal condition to implicitly solve for $Q$. With this, we can now solve this expression for $A$, and substitute in \eqref{eq:terminal_FT-1_AQ} to recover \textit{both} $A$ \textit{and} $Q$. We use root-finding methods to solve for $A$. See Fig.~\ref{fig:parameter_recovery}(bottom) for a numerical example supporting this.

{In the above, we focused on leveraging the terminal condition to identify more parameters than in the infinite horizon setting. We note that generally the now \textit{time-varying} closed-loop dynamics provide us with many more distinct equations. 

\section{Infinite Horizon with Partially Observed State and Control}\label{s:inf+control}

We use the EM algorithm to identify the dynamical system \eqref{eq:system:xt} operating in closed-loop, under the optimal control input \eqref{eq:K_ss} corresponding to the \textit{infinite} horizon stochastic LQR setting. We provide explicit expressions for the E-step and M-step of the EM algorithm, under two different linear Gaussian observation models \eqref{eq:system:yt}:
\begin{subequations}
    \begin{align}
        y_t  & = C x_t + \nu_t, \label{eq:observations:a}\\
        y_t  & =
\begin{bmatrix}
    C_x  x_t\\
    u_t
\end{bmatrix} + 
\begin{bmatrix}
    \nu_{x,t} \\
    \nu_{u,t}
\end{bmatrix} = 
\begin{bmatrix}
    C_x \\
    - K_{ss}
\end{bmatrix} x_t + 
\begin{bmatrix}
    \nu_{x,t}\\
    \nu_{u,t}
\end{bmatrix} \nonumber \\
& \text{ with }
\mathbb{E}\left[
    \begin{bmatrix}
        \nu_{x,t} \\
         \nu_{u,t} 
    \end{bmatrix}  \begin{bmatrix}
        \nu_{x,t} & 
         \nu_{u,t} 
    \end{bmatrix}\right] = 
    \begin{bmatrix}
        V_{x} & 0 \\
        0 & V_{u}
    \end{bmatrix}.
    \label{eq:observations:b}
    \end{align}
\end{subequations}
That is, while the observation model \eqref{eq:observations:a} only has access to partial and noisy measurements of the state, model \eqref{eq:observations:b} additionally has access to noisy measurements of the control input. For convenience, we denote $y_{x,t} = C_x x_t + \nu_{x,t}$ and $y_{u,t} = u_t + \nu_{u,t}$ in \eqref{eq:observations:b}.


We show that while \textit{model (a)} only allows for the identification of the closed-loop $F_{ss}$, \textit{model (b)} provides a further relationship between the inner parameters of the closed-loop $B$ and $K_{ss}$. Technical details are deferred to Appendix \hyperref[app:EM_details]{B}.





\section{Conclusions and Future Works}

\begin{figure}
    \centering
    \includegraphics[width=0.7\columnwidth]{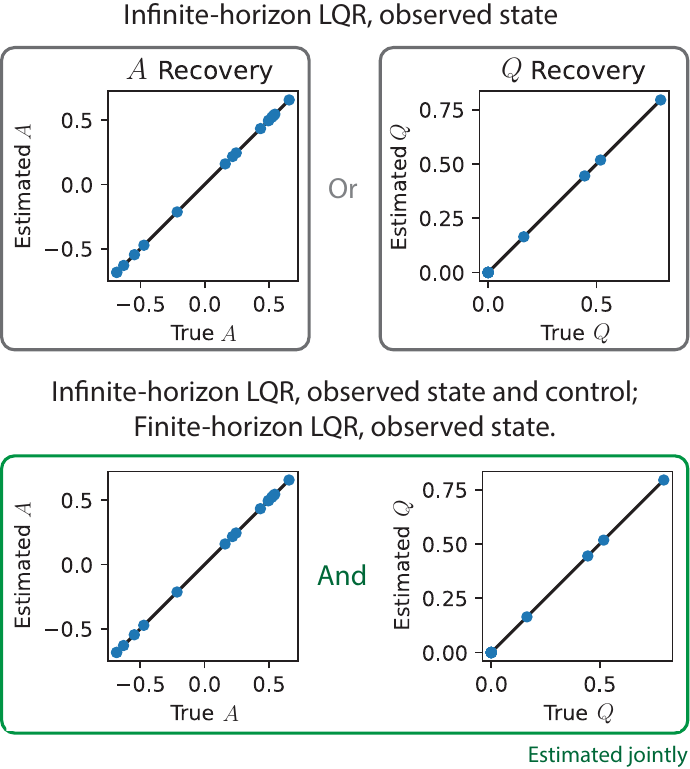}
    \caption{
    Overview of the identification possibilities in each setting. 
    In the infinite horizon, we provide exact and numerical estimation procedures for one parameter given the others; with $B=R=\Id$, we present $A$ given $Q$ and $Q$ given $A$. In the finite horizon, we can leverage terminal conditions to estimate both $A$ and $Q$ (if $Q_T=Q$) simultaneously. In the infinite horizon with both $x_t$ and $u_t$ partially observed, and given $B=R=\Id$, we can estimate $A$ and then estimate $Q$.}
    \label{fig:parameter_recovery}
\end{figure}



We have investigated the identifiability of system and control parameters from partial observations, and provided two extensions that substantially improve identifiability. 
Future theoretical work should explore how estimation errors during system identification influence the recovery of system and control parameters, as well as data-efficient methods for parameter inference in the finite-horizon setting.

These methods provide a promising path towards domain-specific problems of estimating the control properties of partially observed neural systems. For example, the relationship between MEG sensors and the underlying neuromagnetic sources is approximately linear \cite{ilmoniemi2019brain}, suggesting that our methods for system and control recovery may be helpful for understanding human neural control processes from noninvasive recordings.

{\sc Acknowledgements}
The authors would like to thank Henri Schmidt for the helpful discussions.
VG was supported by doctoral scholarships from the Natural Sciences and Engineering Research Council of Canada (NSERC) and the Fonds de recherche du Qu\'ebec – Nature et technologies (FRQNT). 
JA was supported by the Schmidt Science Fellowship from Schmidt Futures and the Rhodes Trust.
HR was supported by the C.V. Starr Fellowship.
JDC was supported by a Vannevar Bush Faculty Fellowship supported by ONR.
JWP was supported by grants from the Simons Collaboration on the Global Brain (SCGB AWD543027), the NIH BRAIN initiative (9R01DA056404-04), an NIH R01 (NIH 1R01EY033064), and a U19 NIH-NINDS BRAIN Initiative Award (U19NS104648).


\section*{APPENDIX}

\subsection{Equivalence between models in marginal likelihood}\label{app:model_equivalence}

Without loss of generality, we omit the dependence on the noise parameters $\theta_n$ in writing our model distributions $p_\theta(x_{0:T}, y_{0:T})$ below. We start by defining the LG-SSM model class for completeness.
\begin{definition}
    A \textit{linear Gaussian state-space model} (LG-SSM) with parameters $\varphi = \{A, C, V, W_0, W\}$ is defined by the joint distribution $q_\varphi(x_{0:T}, y_{0:T})$ over the latent sequence $x_{0:T}$, $x_t \in \R^n$, and observations $y_{0:T}$, $y_t \in \R^d$, satisfying
    \begin{align*}
        q(x_0) = \mathcal{N}(0, W_0), \quad q(x_{t+1} \mid x_t, \varphi) &= \mathcal{N}(x_{t+1} ; A x_t, W), \\
        q(y_t \mid x_t, \varphi) &= \mathcal{N}(y_t ; C x_t, V).
    \end{align*}
\end{definition}


\begin{proposition}\label{prop:appendix:eq_LDS_LQR}
    Let $\mathcal{P} = \{p_\theta : \theta \in \Theta\}$ be the statistical model defined by \eqref{eqs:LQR_problem}-\eqref{eq:system:yt} with parameter space $\Theta$, and let $\mathcal{Q} = \{q_\varphi : \varphi \in \Phi\}$ be an LG-SSM with parameter space $\Phi$. Let our dataset $\mathbf{y}$ be sampled from a model in $\mathcal{P}$, then as $\abs{\mathbf{y}} \to \infty$ 
    \begin{equation*}
        \hat\theta \in \argmax_{\theta \in \Theta} p_\theta\left(\mathbf{y}\right)
        \Longleftrightarrow 
        \{\hat\theta_o, \mathcal{F}(\hat\theta_s)\} \in \argmax_{\varphi \in \Phi} q_\varphi\left(\mathbf{y}\right)
    \end{equation*}
\end{proposition} 

\begin{proof}
    First, we note that the observation models for the two statistical models coincide, that is $p_{\theta_o}(y_{0:T} \mid x_{0:T}) = q_{\theta_o}(y_{0:T} \mid x_{0:T})$ for all $\theta_o \in \Theta_o$. 
    Next for the dynamics distribution, we expand on \eqref{eq:dynamics_prob:lds_equiv} and note that 
    \begin{align*}
        p_{\theta_s}(x_{0:T}) 
        &= p(x_0) \prod_{t=0}^{T-1}  p(x_{t+1} \mid x_t, \mathcal{F}(\theta_s)) \\
        &= q(x_0) \prod_{t=0}^{T-1}  q(x_{t+1} \mid x_t, \mathcal{F}(\theta_s)) = q_{\mathcal{F}(\theta_s)}(x_{0:T}) 
    \end{align*}
    for all $\theta_s \in \Theta_s$. Finally, marginalizing over $x_{0:T}$, we obtain
    \begin{align*}
        p_{\theta}(y_{0:T}) 
        &= \int p_{\theta_o}(y_{0:T} \mid x_{0:T}) p_{\theta_s}(x_{0:T})d x_{0:T}\\
        &= \int q_{\theta_o}(y_{0:T} \mid x_{0:T}) q_{\mathcal{F}\left(\theta_s\right)}(x_{0:T})d x_{0:T} \\
        &= q_{\{\theta_o, \mathcal{F}\left(\theta_s\right)\}}(y_{0:T})
    \end{align*}
    for all $\theta \in \Theta$. We have shown that $\mathcal{P} \subset \mathcal{Q}$ through the mapping $\{\theta_o, \theta_s\} \mapsto \{\theta_o, \mathcal{F}(\theta_s)\}$, and $\Theta \subset \Phi$.
    Now, for the proof,
    \begin{enumerate}
        \item[$\Rightarrow$)] By assumption,
        \begin{equation*}
            \max_{\theta \in \Theta} p_\theta(\mathbf{y}) = p_{\hat\theta}(\mathbf{y}) = q_{\{\hat\theta_o, \mathcal{F}\left(\hat\theta_s\right)\}}(\mathbf{y}).
        \end{equation*}
        To show this is in the maximum set, let us assume there exists another $\varphi \in \Phi$ such that $q_{\varphi}(\mathbf{y}) > q_{\{\hat\theta_o, \mathcal{F}\left(\hat\theta_s\right)\}}(\mathbf{y})$ holds as $\abs{\mathbf{y}} \to \infty$. 
        Then, since the true model is in $\mathcal{P}$, there exists $\tilde\theta$ such that $\varphi = \tilde\theta_o \cup \tilde\theta_s$, which implies $p_{\tilde\theta}(\mathbf{y}) > p_{\hat\theta}(\mathbf{y})$ \textemdash~a contradiction.
        \item[$\Leftarrow$)] This direction follows from from $\Theta \subset \Phi$. If  $\hat\theta_o \cup \mathcal{F}(\hat\theta_s) \in \argmax_{\varphi \in \Phi} q_\varphi(\mathbf{y})$, 
        then 
        $
            q_{\hat\theta_o, \mathcal{F}(\hat\theta_s)} \geq q_{\theta_o, \mathcal{F}(\theta_s)}
        $
        for all $\theta \in \Theta$, which is equivalent to $p_{\hat\theta} \geq p_{\theta}$, as desired.
    \end{enumerate}

    
\end{proof}

\subsection{The EM algorithm}
\label{app:EM_details}
The EM algorithm consists of two steps. In the \textit{E-step}, the current parameter values, denoted by $\theta^{\text{old}}$, are used to compute $\mathcal{Q}(\theta \mid \theta^{\text{old}})$, as defined in \eqref{eq:Q:def}. This yields a function of the parameter set $\theta$, which is then maximized in the \textit{M-step} by setting the appropriate partial derivatives to 0, and solving for each parameter. As the expected log-likelihood $\mathcal{Q}(\theta \mid \theta')$ depends on the observation model, we denote by $\mathcal{Q}_{(a)}$ the expected log-likelihood for the observation model \eqref{eq:observations:a}, and similarly for $\mathcal{Q}_{(b)}$.

\addtolength{\textheight}{-1cm}   
                                  
\paragraph{Partial state observations} In the \textit{E-step} for the observation model \eqref{eq:observations:a}, $y_t = C x_y + \nu_t$, the expected log-likelihood $\mathcal{Q}_{(a)}(\theta \mid \theta^{\text{old}})$ can be expressed as a function of the parameters of interest in $\theta$ and the following \textit{sufficient statistics}
\begin{subequations}
    \begin{align}
        M_{(t_1, t_2)}^{(i)} & \coloneqq \sum_{t=t_1}^{t_2} \big(m_t^{(i)} m_t^{(i)\top} + \Sigma_t^{(i)} \big), \label{eq:M_i}\\
        M_{\Delta}^{(i)} & \coloneqq \sum_{t=1}^{T} \big(m_{t-1}^{(i)} m_t^{(i)\top} + \Sigma_{t-1,t}^{(i)} \big), \label{eq:M_delta_i}\\
        Y^{(i)} & \coloneqq \sum_{t=0}^{T}  y_t^{(i)} y_t^{(i)\top},\quad \tilde{Y}^{(i)} \coloneqq \sum_{t=0}^{T}  m_t^{(i)} y_t^{(i)\top}, \label{eq:Y_suffstats}
    \end{align}
    \label{eq:sufficient_stats_i}
\end{subequations}
for the $i-$th sample,
with $m_t^{(i)}$ and $\Sigma_t^{(i)}$ denoting the mean and the covariance of the posterior distribution over the latent variable at time $t$, that is, $p(x_t \, | \, y_{0:T}^{(i)}) = \mathcal{N}(m_t^{(i)}, \Sigma_t^{(i)})$. These are computed by running the Kalman filter-smoother \cite{Bishop2006} in each independent time series $y_{0:T}^{(i)}$ using the parameter values $\theta^{\text{old}}$, which can be performed \textit{in parallel across datasets}.
We can then pool these sufficient statistics as $\mathcal{M}_{(t_1, t_2)}\coloneqq \sum_{i=1}^N M_{(t_1, t_2)}^{(i)}$, $\mathcal{M}_{\Delta}\coloneqq \sum_{i=1}^N M_{\Delta}^{(i)}$, $\mathcal{Y}\coloneqq \sum_{i=1}^N Y^{(i)}$, and  $\tilde{\mathcal{Y}}\coloneqq \sum_{i=1}^N \tilde{Y}^{(i)}$. 

As for the \textit{M-step}, the explicit parameter updates 
 under the observation model \eqref{eq:observations:a} are then
\begin{align}
\label{eq:Mstep_observation_a}
\begin{split}
        \hat{F}_{ss} & = \mathcal{M}_{\Delta}^{\top} \,\mathcal{M}_{(0,T-1)}^{-1}, \hat{C} =  \tilde{\mathcal{Y}}^{\top} \mathcal{M}_{(0,T)}^{-1}, 
        \\
        \hat{W}_0 & = \frac{1}{N}  \mathcal{M}_{(0,0)},\\ 
       \hat{V} & = \frac{1}{N(T+1)} \Big( \mathcal{Y} - \hat{C} \tilde{\mathcal{Y}} -  \tilde{\mathcal{Y}}^{\top} \hat{C}^{\top} +  \hat{C} \mathcal{M}_{(0,T)} \hat{C}^{\top}\Big),  
       \\
      \hat{W} &  = \frac{1}{NT}  \Big( \mathcal{M}_{(1,T)} -  \mathcal{M}_{\Delta}^{\top} \hat{F}_{ss}^{\top} \\
      &\hspace{1cm}-
\hat{F}_{ss}  \mathcal{M}_{\Delta} + \hat{F}_{ss} \mathcal{M}_{(0,T-1)} \hat{F}_{ss}^{\top} \Big).  
\end{split}
\end{align}
 The explicit updates \eqref{eq:Mstep_observation_a} further speed up the computational routine as the numerical approximation of $\nabla_{\theta} \mathcal{Q}$ is not needed for optimization.

\begin{remark}[Closed-loop identification]
Under the observation model \eqref{eq:observations:a}, $\partial_A \mathcal{Q}_{(a)}  = 0, \partial_B \mathcal{Q}_{(a)} = 0,$ and $\partial_{K_{ss}} \mathcal{Q}_{(a)} = 0$ are redundant. They are all trivially satisfied when $\hat{F}_{ss}$ is as given in \eqref{eq:Mstep_observation_a}, highlighting a lack of identifiability of the inner parameters of the closed-loop.
\end{remark}

\paragraph{Partial state and control observations} The \textit{E-step} for the second observation model \eqref{eq:observations:b} is computed analogously, 
and we introduce the following  partitions of the sufficient statistics \eqref{eq:Y_suffstats}
\begin{subequations}\label{eq:suffcient_stats:partition}
\begin{equation}
    Y^{(i)} 
   = \sum_{t=0}^{T} 
    \begin{bmatrix}
        y_{x,t}^{(i)} \big( y_{x,t}^{(i)}\big)^{\top} & y_{x,t}^{(i)} \big( y_{u,t}^{(i)} \big)^{\top} \\
        y_{u,t}^{(i)} \big(y_{x,t}^{(i)}\big)^{\top} & y_{u,t}^{(i)} \big( y_{u,t}^{(i)} \big)^{\top} 
    \end{bmatrix} \eqqcolon 
     \begin{bmatrix}
            Y_{xx}^{(i)} &  Y_{xu}^{(i)}  \\
            \big(Y_{xu}^{(i)}\big)^{\top} &   Y_{uu}^{(i)} 
        \end{bmatrix}\label{eq:Y_partition}
\end{equation}
and
\begin{equation}
        \tilde{Y}^{(i)} = \sum_{t=0}^{T} m_t^{(i)}
        \begin{bmatrix}
    y_{x,t}^{(i)\top} & y_{u,t}^{(i)\top}
        \end{bmatrix}   \eqqcolon 
        \begin{bmatrix}
            \tilde{Y}_{xx}^{(i)} & \tilde{Y}_{xu}^{(i)}
        \end{bmatrix}, 
\end{equation}
\end{subequations}
and accordingly denote $\mathcal{Y}_{xx} \coloneqq \sum_{i=1}^N Y_{xx}^{(i)}$, $\mathcal{Y}_{xu } \coloneqq \sum_{i=1}^N Y_{xu}^{(i)}$, $\mathcal{Y}_{uu } \coloneqq \sum_{i=1}^N Y_{uu}^{(i)}$, $\tilde{\mathcal{Y}}_{xx}\coloneqq \sum_{i=1}^N \tilde{Y}_{xx}^{(i)}$, and $\tilde{\mathcal{Y}}_{xu }\coloneqq \sum_{i=1}^N \tilde{Y}_{xu}^{(i)}$. 

The parameter updates $\hat{F}_{ss}$, $\hat{W}_0$, and $\hat{W}$
are as provided in \eqref{eq:Mstep_observation_a}. $\hat{C}_x$ is obtained by replacing $\tilde{\mathcal{Y}} \rightarrow \tilde{\mathcal{Y}}_{xx}$ in the $\hat{C}$ update in \eqref{eq:Mstep_observation_a}, and $\hat{V}_x$ is obtained by replacing $\mathcal{Y} \rightarrow \mathcal{Y}_{xx}$ and $\hat{C} \rightarrow \hat{C}_x$ in the $\hat{V}$ update in \eqref{eq:Mstep_observation_a}.
Finally,
\begin{subequations}
    \begin{align}
          \hat{V}_u & = \frac{1}{N(T+1)}  \Big( \hat{K}_{ss} \tilde{\mathcal{Y}}_{xu} +  \tilde{\mathcal{Y}}_{xu}^{\top} \hat{K}_{ss}^{\top} \nonumber \\
          & \hspace{2.8cm} + \hat{K}_{ss} \mathcal{M}_{(0,T)} \hat{K}_{ss}^{\top} + \mathcal{Y}_{uu} \Big),
    \label{eq:Vu_K}\\
       \hat{K}_{ss}  & =  \hat{V}_u \hat{B}^{\top} \hat{W}^{-1} \left(\hat{F}_{ss} - \mathcal{M}_{\Delta}^{\top}\mathcal{M}_{(0,T-1)}^{-1}\right) \nonumber \\
       & \hspace{4.0cm} -   \tilde{\mathcal{Y}}_{xu}^{\top} \mathcal{M}_{(0,T-1)}^{-1}. 
     \label{eq:dL_dK} 
    \end{align}
    \label{eq:grad_Lb_additional}
\end{subequations}
We note that the incorporation of control input observations to the model \eqref{eq:observations:b} makes $\partial_{K_{ss}} \mathcal{Q}_{(b)} = 0$  \textit{not} redundant, providing the additional constraint \eqref{eq:dL_dK}. 

If $B$ is known, the set of equations \eqref{eq:grad_Lb_additional} may be solved for $\hat{V}_u$ and $\hat{K}_{ss}$.
This in turn allows to identify $\hat{A}$ through the identified closed-loop as $\hat{A} = \hat{F}_{ss} + B \hat{K}_{ss}.$ In this case, the identification of the remaining parameters, namely the pair $\{Q, R\}$, maps to an infinite-horizon inverse optimal control problem. Given a stabilizing feedback gain $\hat{K}_{ss}$, as far as we know, there is no analytical solution to the problem of finding all possible pairs $\{ \hat{Q}, \hat{R}\}$ that render $\hat{K}_{ss}$ optimal. Indeed, a manifold of possible $\{\hat{Q}, \hat{R}\}$ pairs yielding $\hat{K}_{ss}$ optimal is expected due to well-established non-uniqueness results \cite{Priess2015}. Additional criteria or constraints can be introduced in the problem to uniquely solve for $\{\hat{Q}, \hat{R}\}$ -- see, e.g. \cite{Priess2015}.


\bibliographystyle{unsrt}
\bibliography{main}

\end{document}